\documentclass[letterpaper, 10 pt, conference]{ieeeconf}  

\IEEEoverridecommandlockouts                              
\makeatletter

\let\proof\@undefined
\let\endproof\@undefined

\makeatother

\usepackage{graphicx}
\usepackage{caption}
\usepackage{subcaption}
\usepackage{amsmath} 
\usepackage{amssymb}  
\usepackage{amsthm}  
\usepackage{bm}
\usepackage{lipsum}
\usepackage[linesnumbered, ruled]{algorithm2e}
\usepackage{color}
\usepackage{enumerate}

\newtheorem{proposition}{Proposition}
\newtheorem{definition}{Definition}

\newtheorem{theorem}{Theorem}

\newcommand{\dist}{\mathrm{dist}}
\newcommand{\cvar}{\mathrm{CVaR}}

\title{\LARGE \bf
Safety-Aware Optimal Control of Stochastic Systems\\ Using Conditional Value-at-Risk}

\author{Samantha Samuelson \and Insoon Yang 
\thanks{This work is supported in part by NSF under  ECCS-1708906 and CNS-1657100. 
S. Samuelson and I. Yang are with the Electrical Engineering Department,
University of Southern California, Los Angeles, CA 90089, USA.
        {\tt\small \{sasamuel,  insoonya\}@usc.edu}}%
}

\begin{document}

\maketitle
\thispagestyle{empty}
\pagestyle{empty}

\begin{abstract}
In this paper, we consider a multi-objective control problem for stochastic systems that seeks to minimize a cost of interest while ensuring safety.
We introduce a novel measure of safety risk 
using the conditional value-at-risk  and a set distance  to formulate a safety risk-constrained optimal control problem.
Our reformulation method using an extremal representation of the safety risk measure
provides a computationally tractable dynamic programming solution. 
A useful byproduct of the proposed solution is the notion of a \emph{risk-constrained safe set}, which is a new stochastic safety verification tool. 
We also establish useful connections between the risk-constrained safe sets and the popular probabilistic safe sets.
The tradeoff between the risk tolerance and the mean performance of our controller is examined through an inventory control problem.
\end{abstract}
\section{Introduction}

Control and verification of safety-critical systems have been 
an important problem in many domains such as 
 air traffic control, autonomous vehicles, robotics, energy systems, and food supply chains.
 A part of such systems can often be modeled as a stochastic system due to the environmental and/or model uncertainty.
To verify that a stochastic system is evolving within a safe range of operation with a pre-specified probability and to construct an associated safety-preserving controller, several reachability-based tools have been developed using  Markov chain approximations
\cite{Hu2005},
Hamilton-Jacobi-Isaacs reachability \cite{Mitchell2005},
barrier certificates \cite{Prajna2007},
dynamic programming for probabilistic safe sets \cite{Abate2008}, and infinite-dimensional linear programming \cite{Kariotoglou2017}, among others.
However, ensuring safety may not be the only objective in practice:  it is also desirable to minimize a cost function of interest by employing an optimal controller.
Even with the aforementioned verification tools, multi-objective stochastic optimal control of safety-critical systems is challenging. 
One safety-oriented suboptimal approach is to use a safe control action whenever the system ventures near the boundary of a probabilistic safe set; otherwise, an optimal control action is used~\cite{Yang2017aut}.  
A lexicographic optimal control approach has been proposed in \cite{Lesser2018} to guarantee that the probability for a system being safe is close to the maximum possible safety probability. 
On the other hand, 
\cite{Svorenova2013} uses linear temporal logic as a constraint to enforce safety with probability 1  in an optimal control problem.

Departing from such probabilistic and temporal logic-based methods,
this paper proposes a risk-based approach to solving safety-aware optimal control of stochastic systems.
Our method employs the conditional value-at-risk (CVaR) and a set distance to measure the risk of a system being unsafe.  By solving an optimal control problem with associated safety risk constraints, we can design a control strategy that minimizes a cost function of interest while limiting the risk of unsafety.
Unlike the probabilistic methods \cite{Yang2017aut}, \cite{Lesser2018}, our proposed method does not require us to separately solve a verification problem to compute probabilistic safe sets or safe control policies. 
In other words, the risk-based approach merges the verification and optimal control procedures into a single step.
A useful byproduct of
our risk-constrained optimal control method is a novel verification tool, called the \emph{risk-constrained safe set}.
Such a set contains all initial states that can be driven to satisfy all safety risk constraints by an admissible control policy.

The contributions of this work can be summarized as follows.
First, we introduce a novel measure of safety risk by using CVaR and 
the distance between the system state and a desired set $A$ for safety. 
This safety risk measure represents the conditional expectation of the distance between the state and $A$ within the $(1-\alpha)$ worst-case quantile of an associated safety loss distribution, where $\alpha \in (0,1)$.
Our method enjoys
an important advantage of CVaR over the value-at-risk (VaR) or the chance constraints that 
CVaR takes into account the possibility of tail events in which safety losses exceed VaR while VaR is incapable of distinguishing situations beyond VaR~\cite{Rockafellar2002a}.
Second, 
we develop a computationally tractable dynamic programming solution through two reformulation procedures. 
The first step reformulates the CVaR constraint in an associated Bellman equation into a tractable expectation constraint.  
The second step removes the minimization problem for computing a set distance from the constraint in the case of finitely supported disturbance distributions.
As a result, the reformulated Bellman equation can be solved by existing convex optimization algorithms when the system dynamics are affine and the cost function is convex.
Third, we establish interesting connections between the proposed risk-constrained safe sets and the popular probabilistic safe sets.
In addition, we propose a simple method to compute the risk-constrained safe sets from the Bellman equation.
The tradeoff between the mean performance and risk tolerance of our controller is also studied through an example of stochastic inventory control.

The remainder of this paper is organized as follows.
Section \ref{safety} introduces the safety risk measure and 
an associated safety-aware optimal control problem.
Its dynamic programming solution is developed in Section \ref{dp}. 
The connections between risk-constrained safe sets and probabilistic safe sets are discussed in Section \ref{rs}.
The performance and risk aversion of the designed controller are demonstrated in Section \ref{ex} through an application to inventory control.


%
%
%
%
%
%
%
%
%
%

%
\section{Set Distance-Based Safety Risk}\label{safety}


Consider the following discrete-time stochastic system:
\begin{equation}\label{sys}
\begin{split}
x_{t+1} &= f(x_t, u_t, w_t),
\end{split}
\end{equation}
where $x_t \in \mathbb{R}^n$ is the system state. 
The control input $u_t$ is assumed to lie in a convex set $\mathbb{U} \subseteq \mathbb{R}^m$.
The stochastic disturbance $w_t \in \mathbb{W} \subseteq \mathbb{R}^l$ is defined on a standard filtered probability space $(\Omega, \mathcal{F}, \{\mathcal{F}_t\}_{t \geq 0},\mathbb{P})$. Note that with the filtration $\{\mathcal{F}_t\}_{t \geq 0}$, $w_t$ is $\mathcal{F}_t$-measurable and thus
so is $x_{t+1}$.
We say that the system is \emph{safe} at stage $t$ if $x_t$ lies in a desired set $A$ for safety, where $A$ is a compact Borel set in $\mathbb{R}^n$. 
The set $A$ represents a safe range of operation in the state space.
Such a setting has been extensively used in the literature of stochastic reachability analysis (e.g., \cite{Abate2008}).
As a first step to unify the safety specification and optimal control of stochastic systems, we propose a novel notion of safety risk using set distance and Conditional Value-at-Risk in the following subsection.

\subsection{Safety Specification Using Conditional Value-at-Risk}

For stochastic systems of the form \eqref{sys}, we can measure the \emph{loss of safety} or the \emph{degree of unsafety} at stage $t$ as the distance between $x_t$ and the set $A$. 
The distance between a point $\bm{x} \in \mathbb{R}^n$ and a set $A \subseteq \mathbb{R}^n$ is defined as follows:
\begin{equation}\label{dist}
\mathrm{dist}(\bm{x},  A) := \inf_{y  \in A} \| \bm{x} - y \|.
\end{equation}
If the system is safe at stage $t$, i.e., $x_t \in A$, 
then the loss of safety $\dist(x_t,  A) = 0$.
On the other hand, when the system is unsafe, i.e., $x_t \notin A$,  the loss of safety increases as $x_t$ moves farther from the desired set $A$.
Note that $\bm{x} \mapsto \mathrm{dist}(\bm{x}, A)$ is convex due to the triangular inequality when the set $A$ is convex. 
In addition, there exists a minimizer if the set $A$ is compact.
We assume that our desired set $A$ for safety is convex and compact.

To quantify safety risk, we adopt Conditional Value-at-Risk (CVaR) among several risk measures that are \emph{coherent} in the sense of Artzner \emph{et al.}~\cite{Artzner1999}.
CVaR measures the expected value conditioned on being within a user-specified percentage ($(1-\alpha) \times 100 \%$) of the worst-case loss scenario.  
CVaR of a random loss $X$ is defined as\footnote{This definition is valid when the probability distribution of $X$ has no atom. For the definition of CVaR in general cases, refer to \cite{Rockafellar2002a}.}
$\cvar_{\alpha}(X) := \mathbb{E} [ X \mid X \geq \mathrm{VaR}_{\alpha} (X)]$
 with $\alpha \in (0,1)$, where the value-at-risk (VaR) of $X$ (with the cumulative distribution function  $F_X$)  is defined as $\mathrm{VaR}_\alpha (X) := \inf \{x \in \mathbb{R} \mid F_X(x) \geq \alpha \}$.
The following \emph{extremal} representation  of CVaR is particularly useful in optimization of CVaR \cite{Rockafellar2000, Rockafellar2002a}:
\begin{equation}\label{cvar}
\cvar_{\alpha}(X) = \min_{z\in \mathbb{R}} \mathbb{E} \bigg [ z + \frac{(X-z)^+}{1-\alpha} \bigg ].
\end{equation}
Suppose that the minimization problem above has a unique optimal solution.  Then, the optimal solution corresponds to VaR at probability~$\alpha$, and CVaR represents VaR plus the expected safety losses exceeding VaR divided by $1-\alpha$.

Using CVaR, we can quantify the risk of the system unsafety at stage $t+1$ given the information collected up to stage $t-1$
 as
\begin{equation} \label{risk}
\begin{split}
&\mathrm{CVaR}_\alpha [\dist (x_{t+1}, A) \mid \mathcal{F}_{t-1}] := 
\\
&\min_{Z\in \mathcal{L}_2 (\Omega, \mathcal{F}_{t-1}, \mathbb{P})} Z +  \mathbb{E} \bigg [\frac{(\dist(x_{t+1}, A)-Z)^+}{1-\alpha} \mid \mathcal{F}_{t-1} \bigg ],
 \end{split}
\end{equation}
which is a random variable adapted to $\mathcal{F}_{t-1}$.
The safety risk  \eqref{risk} measures
the conditional expectation of the distance between $x_{t+1}$ and the desired set $A$  within the $(1-\alpha)$ worst-case quantile of the safety loss distribution.
Note  that we use the conditional version of CVaR (conditioned on $\mathcal{F}_{t-1}$)~\cite{Ruszczynski2006a, Ruszczynski2010}, which is not only practical but also essential to formulate an optimal control problem in a \emph{time-consistent} way as explained in Section \ref{consistency}.

\subsection{Safety-Aware Stochastic Optimal Control}

Our goals in designing a controller are twofold: while controlling the system \eqref{sys}, we want $(i)$ to limit the safety risk \eqref{risk} by a predefined threshold $\delta$ and $(ii)$ to minimize a cost function of interest.
These objectives can be achieved by solving the following risk-constrained stochastic optimal control problem:
\begin{equation} \label{control}
\begin{split}
\min_{\pi \in \Pi} \quad &\mathbb{E}^\pi \bigg [
\sum_{t=0}^{T-1} r(x_t, u_t) + q(x_T)
\bigg ]\\
\mbox{s.t.} \quad& \mathrm{CVaR}_\alpha^\pi [\dist (x_{t+1}, A) \mid \mathcal{F}_{t-1}] \leq \delta, t \in \mathcal{T},
\end{split}
\end{equation}
where 
$r : \mathbb{R}^n \times \mathbb{R}^m \to \mathbb{R}$ and $q : \mathbb{R}^n \to \mathbb{R}$ are a stage-wise and terminal cost function of interest, respectively, and $\mathcal{T} := \{0,1, \cdots, T-1\}$.
Here, the set $\Pi$  of admissible control strategies is given by $\Pi := \{ \pi := (\pi_0, \cdots, \pi_{T-1}) \mid
\pi_t ( \mathbb{U} | h_t) = 1 \; \forall h_t \in H_t \}$, where $H_t$ is the set of \emph{histories} up to stage $t$ whose element is of the form $h_t := (x_0, w_0, \cdots, x_{t-1}, w_{t-1}, x_t)$ and $\pi_t$ is a stochastic kernel from $H_t$ to $\mathbb{U}$.
In addition, $\mathbb{E}^\pi$ and $\cvar_\alpha^\pi$ represent the expectation and CVaR taken with respect to the probability measure induced by a control strategy $\pi$.

The risk tolerance parameter~$\delta$ is nonnegative to be consistent with the nonnegativity of $\dist (x_{t+1}, A)$.
When $\delta = 0$, $x_{t+1}$ must lie in the set $A$ with probability $1$, and thus the risk constraint becomes a hard (deterministic) constraint. The threshold $\delta$ is a user-specified design parameter and has a practical meaning: $\delta$ represents the maximum allowable expected deviation of the state from the set $A$ conditioned on being in the $(1-\alpha)$ worst-case quantile. The effect of $\delta$ on the minimal cost
 depends on problem instances and is studied through an example in Section \ref{ex}. 
 

\section{Dynamic Programming and Convexity}\label{dp}

\subsection{Time-Consistency and Bellman Equation} \label{consistency}

Suppose for a moment that we employ different safety risk constraints of the form $\mathrm{CVaR}_\alpha [\dist (x_{s+1}, A) \mid \mathcal{F}_{0}] \leq \delta$ $\forall s \in \mathcal{T}$. In words, we guarantee the risk constraints assuming that all of them are viewed at stage $0$ with no information.
The tower rule (or the law of total expectation) does not hold for CVaR. Thus,
$\mathrm{CVaR}_\alpha [ \mathrm{CVaR}_\alpha [\dist (x_{t+1}, A) \mid \mathcal{F}_{t-1}] \mid \mathcal{F}_0] \neq  \mathrm{CVaR}_\alpha [\dist (x_{t+1}, A) \mid \mathcal{F}_{0}]$, which implies that the risk constraint may be violated when evaluated  at stage $t$ with information collected up to stage $t-1$. Therefore, this problem formulation is \emph{time-inconsistent}, meaning that an optimal control strategy constructed before or at stage $0$ is no longer optimal when viewed at later stages~\cite{Artzner2007}. 
Dynamic programming is not directly applicable to such a time-inconsistent problem as we cannot break the problem into sub-problems whose optimal solutions can be used to solve the original problem.
There are two main paths to resolve the issue of time-inconsistency. The first is to focus on optimal pre-commitment strategies that are optimal viewed only at stage $0$, and cannot be revised at later stages. Several techniques have been developed to compute an optimal pre-commitment strategy for optimal control of CVaR~\cite{Bauerle2011, Borkar2014, Chow2015, Haskell2015, Pflug2016, Miller2017}.
The second strategy is to employ time-consistent dynamic risk measures that guarantee the time-consistency of the control problem~\cite{Artzner2007, Ruszczynski2010, Chow2013, Cavus2014}. 
This approach requires special care when interpreting the practical meaning of such risk measures as they are usually defined as a composition of multiple conditional risk mappings. 

Our problem formulation is closely related with the second approach: our conditional version of CVaR \eqref{risk} ensures the time-consistency of the optimal control problem \eqref{control}.
By solving \eqref{control}, a control strategy is designed offline before stage $0$ to satisfy the risk constraint $\mathrm{CVaR}_\alpha [\dist (x_{t+1}, A) \mid \mathcal{F}_{t-1}] \leq \delta$ at stage $t$ using information gathered up to stage $t-1$.
Thus, the designed control strategy ensures the
risk constraint when viewed and evaluated at stage $t$. 
To check the applicability of dynamic programming, we decompose \eqref{control} into multiple sub-problems whose optimal solutions can be used to design an optimal strategy for \eqref{control}.
We define the value function associated with \eqref{control} as follows:
\begin{equation}\label{value}
\begin{split}
v_t(\bm{x}) := \inf_{\pi  \in \Pi}  \: &\mathbb{E}^\pi \bigg [ \sum_{s=t}^{T-1} r(x_s, u_s) + q(x_T) \mid x_t = \bm{x}
\bigg ]\\
\mbox{s.t} \:\: & \mathrm{CVaR}_\alpha^\pi [\dist (x_{s+1}, A) \mid \mathcal{F}_{s-1}] \leq \delta,   s \in \mathcal{T}_t,
\end{split}
\end{equation}
which represents the minimum expected cost-to-go given the safety risk constraints are satisfied for all stages from $t$ to $T-1$, where $\mathcal{T}_t := \{t, t+1, \cdots, T-1\}$. 
We now use backward induction to confirm that 
the subproblems \eqref{value} can be used to solve \eqref{control}. Note that $v_T$ is given by $q$.
Suppose now that $v_{t+1}$ is known.
Then, 
\begin{equation}\label{original}
\begin{split}
v_t(\bm{x}) = \inf_{\bm{u} \in \mathbb{U}} \;\; &\mathbb{E} [ r(\bm{x}, \bm{u}) + v_{t+1} (f(\bm{x}, \bm{u}, w_t)) ] \\
\mbox{s.t.} \;\; &\cvar_\alpha [\dist (x_{t+1}, A) \mid \mathcal{F}_{t-1} ] \leq \delta
\end{split}
\end{equation}
 because  the risk constraints for $s \in \mathcal{T}_{t+1}$ are considered in the optimization problem \eqref{value} for $v_{t+1}$.
However, the Bellman equation \eqref{original} involves a triple-level minimization problem with $(i)$ the outer minimization problem over $\bm{u}$,  $(ii)$ the middle-level minimization problem \eqref{risk} for CVaR and $(iii)$ the inner minimization problem \eqref{dist} for the distance function.
We can significantly simplify the Bellman equation by reformulating the CVaR constraints as expectation constraints.


\begin{theorem}[Bellman equation I] \label{reform1}
The value function defined in \eqref{value} satisfies the following Bellman equation:
\begin{equation} \label{bellman}
\begin{split}
v_t(\bm{x}) = 
 \inf_{\bm{u} \in \mathbb{U}, z \in \mathbb{R}} \;\; &r(\bm{x}, \bm{u}) + \mathbb{E} [ v_{t+1}(f(\bm{x}, \bm{u}, w_t))  ]\\
\mathrm{s.t.} \;\; &\mathbb{E} \bigg [ z +
\frac{(\dist (f(\bm{x}, \bm{u}, w_t), A) - z)^+}{1-\alpha} \bigg ]
 \leq \delta
\end{split}
\end{equation}
for $t \in \mathcal{T}$
with $v_T(\bm{x}) = q(\bm{x})$.
\end{theorem}
\begin{proof}
Using the dynamic programming principle, we have
\begin{equation}\nonumber
\begin{split}
v_t(\bm{x}) = \inf_{\bm{u} \in \mathbb{U}} \quad &r(\bm{x}, \bm{u}) + \mathbb{E} [ v_{t+1}(f(\bm{x}, \bm{u}, w_t)) ]\\
\mbox{s.t} \quad & \mathrm{CVaR}_\alpha [\dist (f(\bm{x}, \bm{u}, w_t), A)] \leq \delta,
\end{split}
\end{equation}
which is equivalent to \eqref{original}.
We denote the right-hand side of \eqref{bellman} as $\hat{v}_t(\bm{x})$ and show that $\hat{v}_t = v_t$.
To show that $\hat{v}_t(\bm{x}) \leq v_t (\bm{x})$ fixing an arbitrary $\bm{x} \in \mathbb{R}^n$, we first note that
for any $\epsilon > 0$ there exists $\bm{u}^\star \in \mathbb{U}$ such that
\begin{equation} \label{eps2}
v_t(\bm{x}) + \epsilon > r(\bm{x}, \bm{u}^\star)
 +\mathbb{E}[ v_{t+1}(f(\bm{x}, \bm{u}^\star, w_t)) ]
\end{equation}
and
\[
\mathrm{CVaR}_\alpha [\dist (f(\bm{x}, \bm{u}^\star, w_t), A) ] \leq \delta.
\]
By the extremal representation \eqref{cvar} of CVaR, the second inequality is equivalent to
\[
\min_{z \in \mathbb{R}} \mathbb{E} \bigg [
z + \frac{(\dist(f(\bm{x}, \bm{u}^\star, w_t), A)- z)^+}{1-\alpha}
\bigg ] \leq \delta,
\]
which implies that  there exists $z^\star \in \mathbb{R}$ such that
\[
\mathbb{E} \bigg [ z^\star + \frac{(\dist(f(\bm{x}, \bm{u}^\star, w_t), A)- z^\star)^+}{1-\alpha}  \bigg ]
\leq \delta.
\]
Combining this with the inequality \eqref{eps2}, we have
\begin{equation} \nonumber
\begin{split}
v_t(\bm{x}) + \epsilon &> \\
\inf_{\bm{u} \in \mathbb{U}, z \in \mathbb{R}} \;
&r(\bm{x}, \bm{u}) + \mathbb{E} [ v_{t+1}(f(\bm{x}, \bm{u}, w_t))] \\
\mathrm{s.t.} \; & 
\mathbb{E} \bigg [ z + \frac{(\dist(f(\bm{x}, \bm{u}, w_t), A)- z)^+}{1-\alpha}  \bigg ] \leq \delta.
\end{split}
\end{equation}
Letting $\epsilon \to 0$, we have $\hat{v}_t(\bm{x}) \leq v(\bm{x})$.

We now show that $\hat{v}_t(\bm{x}) \geq v(\bm{x})$.
For any $\epsilon > 0$, there exists $(\hat{\bm{u}}, \hat{\bm{z}}) \in \mathbb{U} \times \mathbb{R}$ such that
\[
\hat{v}_t(\bm{x}) + \epsilon >
r(\bm{x}, \hat{\bm{u}}) + \mathbb{E}[ v_{t+1}( f(\bm{x}, \hat{\bm{u}}, w_t)) ]
\]
and 
\[
\mathbb{E} \bigg [ \hat{z} + \frac{(\dist(f(\bm{x}, \hat{\bm{u}}, w_t), A) - \hat{z})^+}{1-\alpha} \bigg ] \leq \delta.
\]
Due to the extremal formula \eqref{cvar} of CVaR, the second inequality implies that 
\[
\cvar_\alpha [ \dist( f(\bm{x}, \hat{\bm{u}}, w_t), A)  ] \leq \delta.
\]
Therefore,
$\hat{v}_t(\bm{x}) + \epsilon > v_t(\bm{x})$, which implies that $\hat{v}_t(\bm{x}) \geq {v}_t(\bm{x})$ as $\epsilon \to 0$.
\end{proof}
%
%
%
%
%

The minimization problem in the reformulated Bellman equation \eqref{bellman} is a computationally tractable stochastic program while the original Bellman equation \eqref{original} involves a nontrivial CVaR constraint. 
A similar reformulation approach has been proposed by Krokhmal \emph{et al.}~\cite{Krokhmal2002} in the context of single-stage optimization with CVaR constraints. Unlike their method based on the Karush-Kuhn-Tucker conditions, however, our proof does not assume the existence of an optimal solution $\bm{u}^{opt}$ or the convexity of the objective and constraint functions.
In other words, the proposed method not only yields a computationally tractable version of the Bellman equation but also broadens the applicability of the efficient reformulation method in \cite{Krokhmal2002} for CVaR-constrained optimization.
We will further enhance the computational tractability of \eqref{bellman} 
 in Section \ref{empirical}.

\subsection{Convexity of Value Functions}

%
%
%
%
%
%

We now provide conditions under which the stochastic program in the Bellman equation \eqref{bellman} and the value function $v_t$ are convex.

\begin{proposition} \label{convex}
Suppose that $(\bm{x}, \bm{u}) \mapsto f(\bm{x}, \bm{u}, \bm{w})$ is affine on $\mathbb{R}^n \times \mathbb{U}$ for each $\bm{w} \in \mathbb{W}$,  $r$ is convex on $\mathbb{R}^n \times \mathbb{U}$, and $q$ is convex on $\mathbb{R}^n$. Then, the value function $v_t$ is convex on $\mathbb{R}^n$ for all $t \in \bar{\mathcal{T}}$.
\end{proposition}

\begin{proof}
We use mathematical induction backward in time.
At stage $T$, $v_T = q$ is convex on $\mathbb{R}^n$.
Suppose now that $v_{t+1}$ is convex on $\mathbb{R}^n$.
At stage $t$, fix two arbitrary states $\bm{x}^1,\bm{x}^2 \in \mathbb{R}^n$.
Due to the Bellman equation \eqref{bellman},
for any $\epsilon > 0$ there exists $(\bm{u}^i, z^i) \in \mathbb{U} \times \mathbb{R}$, $i=1,2$, such that 
\begin{equation}\label{eps_ineq}
v_t(\bm{x}^i) + \epsilon >  r(\bm{x}^i, \bm{u}^i) + 
\mathbb{E} [ v_{t+1} ( f(\bm{x}^i, \bm{u}^i, w_t)) ]
\end{equation}
and
\[
\mathbb{E} \bigg [
z^i + \frac{(\dist(f(\bm{x}^i, \bm{u}^i, w_t), A)  - z^i)^+}{1-\alpha} 
\bigg ] \leq \delta.
\]
Let $\bm{x}^\lambda := \lambda \bm{x}^1 + (1-\lambda) \bm{x}^2 \in \mathbb{R}^n$ and $(\bm{u}^\lambda, z^\lambda) := \lambda (\bm{u}^1, z^1) + (1-\lambda) (\bm{u}^2, z^2) \in \mathbb{U} \times \mathbb{R}$.
We first note that $(\bm{x}, \bm{u}) \mapsto \dist( g(\bm{x}, \bm{u}, w_t), A)$ is convex for each $w_t \in \mathbb{W}$ since $\bm{x} \mapsto \dist( \bm{x}, A)$ is convex and $(\bm{x}, \bm{u}) \mapsto f(\bm{x}, \bm{u}, w_t)$ is affine. 
Thus,
$(\bm{x}, \bm{u}, z) \mapsto (\dist( f(\bm{x}, \bm{u}, w_t), A) - z)^+$ is convex for each $w_t \in \mathbb{W}$ as $a \mapsto (a)^+$ is a convex increasing function.
Therefore, 
\begin{equation} \nonumber
\begin{split}
&\mathbb{E} \bigg [
z^\lambda + \frac{(\dist (f(\bm{x}^\lambda, \bm{u}^\lambda, w_t), A)  - z^\lambda)^+}{1-\alpha} 
\bigg ]\\
& \leq \lambda \mathbb{E} \bigg [
z^1 + \frac{(\dist (f(\bm{x}^1, \bm{u}^1, w_t), A)  - z^1)^+}{1-\alpha} 
\bigg ]\\
&+ (1-\lambda) \mathbb{E} \bigg [
z^2 + \frac{(\dist (f(\bm{x}^2, \bm{u}^2, w_t), A)  - z^2)^+}{1-\alpha} 
\bigg ]\\
&\leq \lambda \delta + (1-\lambda) \delta = \delta.
\end{split}
\end{equation}
This implies that $(\bm{u}^\lambda, z^\lambda)$ is a feasible solution of the minimization problem in \eqref{bellman} with $\bm{x} = \bm{x}^\lambda$.
Thus, 
\[
v_t(\bm{x}^\lambda) \leq r(\bm{x}^\lambda, \bm{u}^\lambda)
+ \mathbb{E}[ v_{t+1}(f(\bm{x}^\lambda, \bm{u}^\lambda, w_t)) ].
\]
Since $r$ is convex on $\mathbb{R}^n \times \mathbb{U}$
and $(\bm{x}, \bm{u}) \mapsto v_{t+1}(f(\bm{x}, \bm{u}, w_t))$ is convex due to the induction hypothesis, 
we have
\begin{equation} \nonumber
\begin{split}
&v_t(\bm{x}^\lambda) \leq \lambda r(\bm{x}^1, \bm{u}^1) + (1-\lambda) r(\bm{x}^2, \bm{u}^2) \\
&+ \mathbb{E} [\lambda v_{t+1} (f (\bm{x}^1, \bm{u}^1, w_t))
+ (1-\lambda) v_{t+1} (f (\bm{x}^2, \bm{u}^2, w_t))].
\end{split}
\end{equation}
Combining this inequality with \eqref{eps_ineq}, we finally obtain
\[
v_t(\bm{x}^\lambda) < \lambda v_t(\bm{x}^1) + (1-\lambda) v_t(\bm{x}^2) + \epsilon.
\]
Letting $\epsilon \to 0$, we conclude that $\bm{x} \mapsto v_t(\bm{x})$ is convex on $\mathbb{R}^n$. This completes our inductive argument.
\end{proof}

Since $v_t$ is convex for all $t \in \mathcal{T}$ under the conditions in Proposition \ref{convex}, the objective function of the stochastic program in the Bellman equation \eqref{bellman} is convex. The constraint is also convex since $\bm{u} \mapsto f(\bm{x}, \bm{u}, w_t)$ is affine for each $(\bm{x}, w_t)$ and $\bm{x} \mapsto \dist( \bm{x}, A)$ is convex.
Therefore, the stochastic program in \eqref{bellman} is convex.
This convexity is also used in our numerical experiments in Section \ref{ex} to approximate $v_t$ as
 the convex envelope of $v_t$ discretized over~$\bm{x}$.

\subsection{Finitely Supported Disturbance Distributions} \label{empirical}

We now consider the case of finitely supported disturbance distributions. This case is practically important as most empirical distributions directly obtained from data have a finite support. 
Furthermore, the popular sample average approximation (e.g., \cite{Robinson1996}) reduces the control problem \eqref{control} with  an infinite support to the case of
finitely supported disturbance distributions.
Suppose that the support $\mathbb{W}$ of the disturbance distribution is given by
\begin{equation}\label{support}
\mathbb{W} := \{\bm{w}^{(i)} \in \mathbb{R}^l \: | \: i= 1, \cdots, N\},
\end{equation}
which is a finite set.
In this case, we can further simplify the Bellman equation \eqref{bellman} as the following deterministic optimization problem by removing the set distance function from the constraints.

\begin{theorem}[Bellman equation II]
Suppose that the disturbance distribution has a finite support of the form \eqref{support}.
Then, the Bellman equation \eqref{bellman} is equivalent to
\begin{equation} \nonumber
\begin{split}
v_t(\bm{x}) = \qquad \;\; &\\
\inf_{\bm{u} \in \mathbb{U}, y \in A^N, z \in \mathbb{R}} \; &r(\bm{x}, \bm{u}) + \frac{1}{N} \sum_{i=1}^N v_{t+1}(f(\bm{x}, \bm{u}, \bm{w}^{(i)}))\\
\mathrm{s.t.} \; & z + \frac{
\sum_{i=1}^N (\| f(\bm{x}, \bm{u}, \bm{w}^{(i)}) - y^{(i)} \| - z)^+}{N(1-\alpha)} 
\leq \delta
\end{split}
\end{equation}
for $t \in \mathcal{T}$
with $v_T(\bm{x}) = q(\bm{x})$.
\end{theorem}
\begin{proof}
Let $\tilde{v}_t(\bm{x})$ be the right-hand side of the equality above.
With the support \eqref{support}, the constraint in the Bellman equation \eqref{bellman} can be rewritten as
\begin{equation} \nonumber
\begin{split}
z +
\frac{\sum_{i=1}^N(\dist (f(\bm{x}, \bm{u}, \bm{w}^{(i)}), A) - z)^+}{N(1-\alpha)} 
 \leq \delta,
\end{split}
\end{equation}
where $\dist (f(\bm{x}, \bm{u}, \bm{w}^{(i)}), A)
= \min_{y^{(i)} \in A} \| f(\bm{x}, \bm{u}, \bm{w}^{(i)}) - y^{(i)} \|$ since $A$ is compact and convex.
Using an argument similar to the proof of Theorem \ref{reform1}, we have that $\tilde{v}_t = v_t$.
\end{proof}

Under the conditions in Proposition \ref{convex},
the simplified Bellman equation involves a deterministic convex program, which can be efficiently solved by several existing convergent algorithms.
The dimension of its optimization variable linearly increases the cardinality $N$ of the support $\mathbb{W}$.
It is worth mentioning that our focus is not to resolve the fundamental scalability issue in dynamic programming: the computational complexity of our approach scales exponentially with the dimension $n$ of state space as in standard dynamic programming. 
The major advantage of our method is to reformulate the triple-level optimization problem in the original Bellman equation \eqref{original} as a tractable single-level optimization problem.

\section{Risk-Constrained Safe Sets} \label{rs}

So far, we have viewed the CVaR-based safety risk as a constraint of an optimal control problem. 
In this section, 
we illustrate how the safety risk can be used to verify the safety of stochastic systems.

\subsection{Connection to Probabilistic Reachability Analysis}

We begin by establishing a few interesting connections between our risk-based approach and the probabilistic safety/reachability specifications.
To verify that a stochastic system starting from a particular initial point can be controlled to operate in a safe range $A$ with a pre-specified probability $\alpha$,
one can use the \emph{probabilistic safe set}, defined as
\begin{equation}\nonumber
\begin{split}
S_\alpha (A) := \{&\bm{x} \in \mathbb{R}^n \mid 
\exists \pi \in \Pi \mbox{ s.t. }x_0 = \bm{x},\\ 
& \mathbb{P}^\pi (x_{t} \in A, \: t =1, \cdots, T ) \geq \alpha\}.
\end{split}
\end{equation}
If $x_0 \in S_\alpha (A)$, then there exists a control strategy that guarantees the system safety with probability greater than or equal to $\alpha$.
Dynamic programming-based tools to compute such probabilistic safe sets have been developed for stochastic hybrid systems (SHS) \cite{Abate2008, Summers2010, Ding2013}, partially observable SHS \cite{Lesser2017}, and stochastic systems under distributional ambiguity \cite{Yang2017aut}.
Departing from these tools for probabilistic safe sets, our risk-constrained method provides the following novel safe sets that can also be used for safety specification and verification:

\begin{definition}[Risk-constrained safe set]
We define the \emph{risk-constrained safe set} for $A$ as
\begin{equation} \nonumber
\begin{split}
RS_{\alpha, \delta} ( A) := \{ &\bm{x} \in \mathbb{R}^n \mid \exists \pi \in \Pi  \mbox{ s.t. } x_0 = \bm{x}, \\ 
&\cvar_\alpha^\pi [\dist (x_{t+1}, A) \mid \mathcal{F}_{t-1} ] \leq \delta, t \in \mathcal{T} \}
\end{split}
\end{equation}
for some $\alpha \in (0,1)$ and $\delta \geq 0$.
\end{definition}

In words, whenever $x_0 \in RS_{\alpha, \delta} (A)$, 
we can control the system to satisfy the CVaR-based safety constraint  for all stages. 
We will introduce a method to compute the risk-constrained safe sets in the next subsection.
Before this, we take a close look at the CVaR constraint to 
relate $RS_{\alpha, \delta}(A)$ with $S_\alpha(A)$.
Our first observation is that when $\delta = 0$,
$\dist (x_{t+1}, A) = 0$ with probability 1
because $(i)$ its $(1-\alpha)$ worst-case quantile is less than or equal to zero and $(ii)$ the distance is greater than or equal to zero by definition.
This observation leads to the following proposition:

\begin{proposition} \label{hard}
If the risk threshold parameter $\delta = 0$, then 
the risk-constrained safe set $RS_\alpha (A)$ 
is a subset of the probabilistic safe set $S_\alpha (A)$. Furthermore,
\[
RS_{\alpha, 0} (A) = S_1 (A) \subseteq S_\alpha (A) \quad \forall \alpha \in (0,1).
\]
\end{proposition}
\begin{proof}
Fix the initial state $x_0$ as $\bm{x} \in RS_{\alpha, 0} (A)$
for  $\delta = 0$. Then, there exists a strategy $\pi \in \Pi$ such that for each $\alpha \in (0,1)$, $\mathbb{P}^\pi ( \dist(x_{t+1}, A) = 0 \mid x_t, \cdots, x_1) = \mathbb{P}^\pi ( x_{t+1} \in A \mid x_t, \cdots, x_1) = 1$ for all $t \in \mathcal{T}$.
Due to the chain rule, 
\begin{equation} \nonumber
\mathbb{P}^\pi \Big ( \bigcap_{t = 1}^{T} \{x_t \in A \} \Big )= \prod_{t = 1}^{T} \mathbb{P}^\pi \Big ( x_{t} \in A \mid \bigcap_{s =1}^{t-1} \{ x_s \in A \} \Big ) = 1,
\end{equation}
which implies that 
$\bm{x} \in S_1 (A)$.
The reverse of the aforementioned argument is also valid.
Thus, if $\bm{x} \in S_1(A)$, then $\bm{x} \in RS_{\alpha, 0}(A)$.
Note also that $S_1(A) \subseteq S_\alpha (A)$ for any $\alpha \in (0, 1)$, the statement in the proposition holds.
\end{proof}

Proposition \ref{hard} implies that $RS_{\alpha, 0} (A)$ can be used for very conservative decision-making in terms of safety via hard constraints. 
When $\delta > 0$, we have another interesting connection between risk-constrained and probabilistic safe sets as follows:
\begin{proposition}\label{relax}
Let $A_\delta := \{ \bm{x} \in \mathbb{R}^n \mid \dist (\bm{x}, A) \leq \delta \}$ for any $\delta > 0$. Then, the risk-constrained safe set $RS_{\alpha, \delta}(A)$ is a subset of the probabilistic safe set $S_{\alpha^T} (A_\delta)$, i.e.,
\[
RS_{\alpha, \delta}(A) \subseteq S_{\alpha^T} (A_\delta).
\]
\end{proposition}
\begin{proof}
Choose an arbitrary initial state $\bm{x}$ from $RS_{\alpha,\delta} (A)$. Then, there exists a control policy $\pi \in \Pi$ such that for each $\alpha \in (0,1)$,
$\mathbb{P}^\pi ( \dist(x_{t+1}, A) \leq \delta \mid x_t, \cdots, x_1) = \mathbb{P}^\pi ( x_{t+1} \in A_\delta \mid x_t, \cdots, x_1) \geq  \alpha$ for all $t \in \mathcal{T}$. The chain rule for conditional probability implies that
$\mathbb{P}^\pi  ( \bigcap_{t = 1}^{T} \{x_t \in A_\delta \}  )= \prod_{t = 1}^{T} \mathbb{P}^\pi  ( x_{t} \in A_\delta \mid\bigcap_{s =1}^{t-1} \{ x_s \in A_\delta \} ) \geq \alpha^T$.
Thus, $\bm{x} \in S_{\alpha^T} (A_\delta)$ for each $\alpha \in (0,1)$.
\end{proof}

Due to the definition of set distance-based safety risk, 
Proposition \ref{relax} compares $RS_{\alpha, \delta}(A)$ with the probabilistic safe set for a relaxed desirable set $A_\delta$.
To compare with $S_{\alpha}(A)$ instead of $S_{\alpha}(A_{\delta})$,
it is often useful to consider $RS_{\alpha, \delta}(A_{-\delta})$, which is contained in $S_{\alpha^T}(A)$, where
$A_{-\delta} := \{\bm{x} \in \mathbb{R}^n \mid  
\dist (\bm{x}, A^c ) \geq \delta \}$ for $\delta > 0$.


%
%


\subsection{From Value Functions to Risk-Constrained Safe Sets}

We now propose a simple approach to computing the risk-constrained safe sets by using the value function of \eqref{control}.
The key idea is that $\bm{x} \in RS_{\alpha, \delta}(A)$ 
if the control problem \eqref{control} with $x_0 = \bm{x}$ has a non-empty feasible set.

\begin{theorem}\label{comp}
Suppose that
\[
r(\bm{x}, \bm{u}) < +\infty \;\; \forall \bm{u} \in \mathbb{U}, \quad  q(\bm{x}) < +\infty
\]
for each $\bm{x} \in \mathbb{R}^n$.
Then, the risk-constrained safe set can be computed as
\[
RS_{\alpha, \delta} (A) = \{ \bm{x} \in \mathbb{R}^n \mid
v_0(\bm{x}) < +\infty \}.
\]
\end{theorem}

\begin{proof}
Fix $\bm{x} \in RS_{\alpha, \delta} (A)$.
Then, there exists a control policy $\hat{\pi} \in \Pi$ such that 
$\cvar_\alpha^{\hat{\pi}} [\dist (x_{t+1}, A) \mid \mathcal{F}_{t-1} ] \leq \delta$, $t \in \mathcal{T}$, where $x_0 = \bm{x}$.
Suppose that $v_0 (\bm{x}) = +\infty$. Since $\hat{\pi}$ satisfies all the risk constraints,
$\mathbb{E}^{\hat{\pi}}[ \sum_{t=0}^{T-1} r(x_t, u_t) + q(x_T) \mid x_0 = \bm{x}] \geq v_0 (\bm{x}) = +\infty$.
However, $r(x_t, u_t) < +\infty$ $\forall t \in \mathcal{T}$ and $q(x_T) < +\infty$ under the policy $\hat{\pi}$. 
This is a contradiction and thus $v_0 (\bm{x}) < +\infty$.

We now choose an arbitrary initial state $x_0 = \bm{x}$ such that $v_0(\bm{x}) < +\infty$.
For any $\epsilon > 0$, there exists $\tilde{\pi} \in \Pi$ such that
$\mathbb{E}^{\tilde{\pi}} [ \sum_{t=0}^{T-1} r(x_t, u_t) + q(x_T) \mid x_0 = \bm{x} ] < v_t(\bm{x}) + \epsilon < +\infty$.
Thus, 
 $\cvar_\alpha^{\tilde{\pi}} [ \dist (x_{t+1}, A) \mid \mathcal{F}_{t-1}] \leq \delta$, $t \in \mathcal{T}$, where $x_0 = \bm{x}$.
 This implies that $\bm{x} \in RS_{\alpha, \delta}(A)$. 
\end{proof}
 
By Theorem \ref{comp}, 
we can use our dynamic programming solution of \eqref{control} in two useful ways. 
First, we can verify whether a given initial state $x_0$ will satisfy all the safety risk constraints by checking the value $v_0(x_0)$.
Second, we can explicitly construct an optimal risk-averse policy $\pi^{opt}$ of \eqref{control} by solving associated Bellman equations backward in time.
In particular, under the measurable selection condition (e.g., \cite{Hernandez2012}), the Bellman equation admits an optimal solution $\bm{u}^{opt}$ for each $(t, \bm{x})$ and thus one can construct a non-randomized Markov policy, which is optimal, by letting $\pi_t^{opt}(\bm{x}) := \bm{u}^{opt}$.

\section{Application to Inventory Control}\label{ex}

\begin{figure}[tb]
\begin{center}
\includegraphics[width=2.3in]{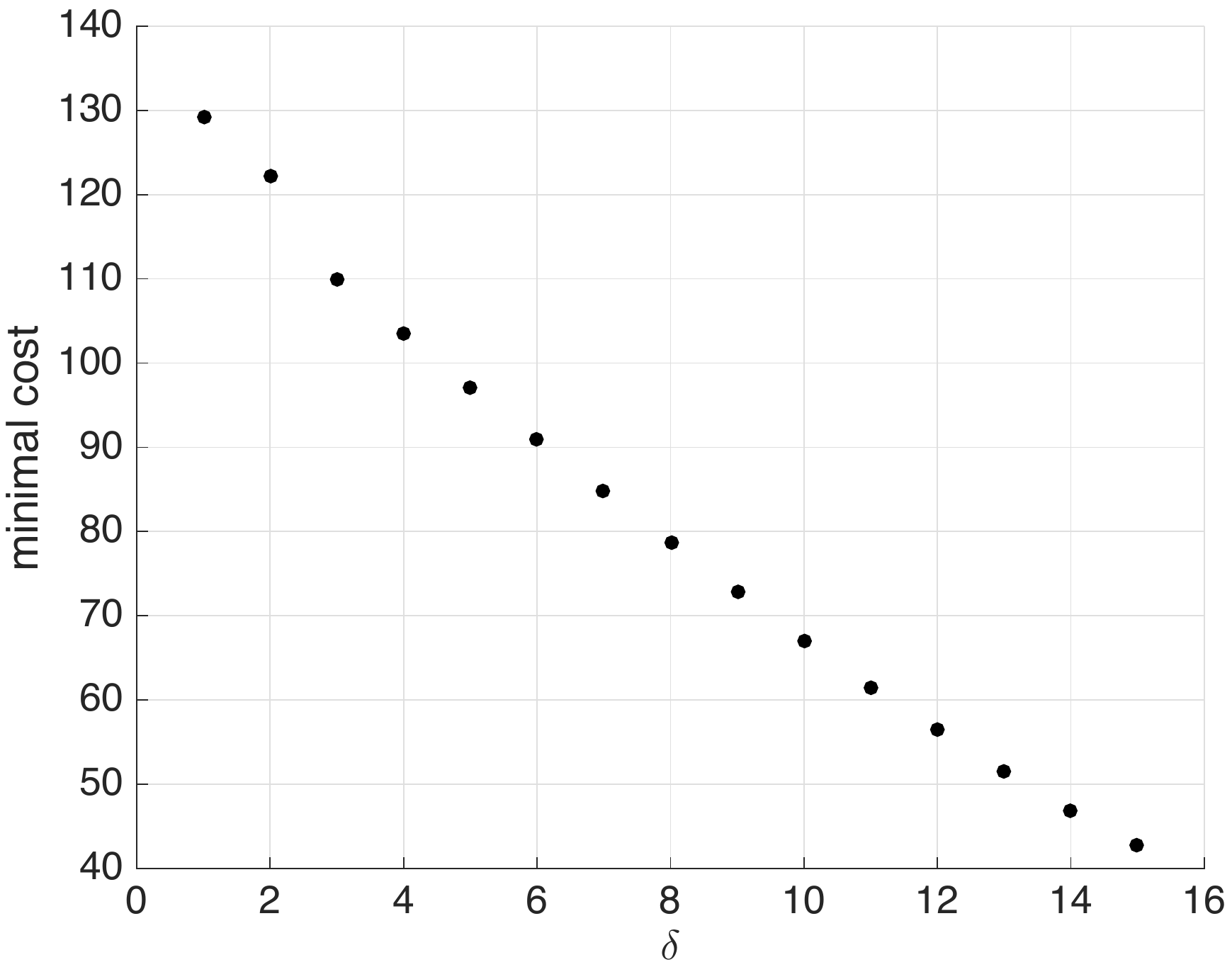}
\caption{Minimal expected cost over independent simulations for several different parameters $\delta$.
}  
\label{fig:tradeoff}    
\end{center}           
\end{figure}

To demonstrate the advantages of using our  approach in a realistic problem, we examine an inventory control model. We define the state evolution function as 
\[
x_{t+1} = x_t + u_t - w_t,
\]
where $u_t$ is the quantity ordered/received at stage $t$, $w_t$ is the demand at stage $t$, and $x_t$ is the current inventory level. The control is bounded by $u \in [0, 32]$, and we use a time horizon of one week, i.e., $\mathcal{T} := \{0,1, \cdots, 7\}$.  Any demand that is left unsatisfied is backlogged for the next stage, which is represented as a negative state value. We define the stage-wise cost as 
\[
r(x_t, u_t, w_{t}) := c_o(x_t + u_t - w_{t})^+ + c_u (w_{t} - x_t - u_t)^+,
\]
where $c_o = 1$ represents the holding or storage cost and $c_u = 1$ represents the cost of lost sales due to unavailable inventory.
The desired set for safety is chosen as $A = [0, 100]$.  
We use $N=40$ samples of $w_t$, generated from the distribution $w_t \sim \mathcal{N}(20,6)$.

We first examine the tradeoff between the mean performance and the risk tolerance of our controller.  
Fig. \ref{fig:tradeoff} shows the mean total cost over independent simulations for several different risk tolerance values $\delta$ when $\alpha=0.90$.  
As the constraint is tightened by decreasing $\delta$, the mean total cost increases. Having a larger $\delta$ corresponds to a larger allowable deviation from the desired set $A$, so the total cost decreases.  The choice of optimal $\delta$ thus depends on the designer's preference for either a low-risk or a low-cost controller.  
Fig.~\ref{fig:value} plots the value function at stage $t=0$, for various values of $\delta$. Here, we can see how higher values of $\delta$ generate lower expected costs.  This figure also shows another effect of tightening the CVaR constraint: the set  $RS_{\alpha,\delta}(A)$ of feasible initial states  becomes smaller as $\delta$ decreases.  

\begin{figure}[tb]
\begin{center}
\includegraphics[width=2.5in]{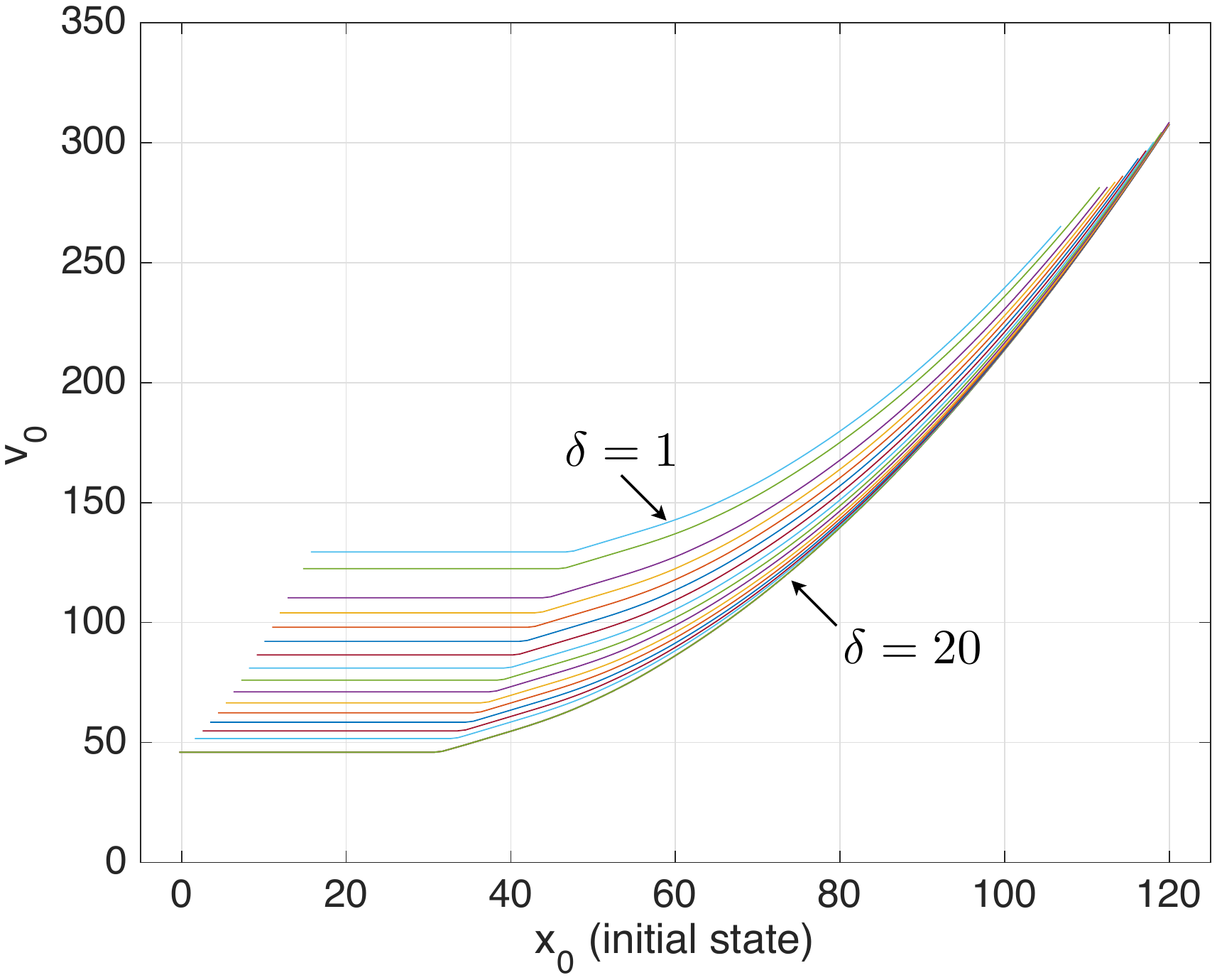}
\caption{Value function $v_0$ (minimal expected cost) for $\delta = 1, \cdots, 20$ on $\{x_0 : v_0(x_0) < +\infty\}$.
}  
\label{fig:value}    
\end{center}           
\end{figure}

Let $RS_{\alpha,\delta, t}(A)$ be the time-dependent risk-constrained safe set initialized at stage $t$, that is the set of $x_t$'s for which the CVaR constraint can be satisfied at all future times.  
As shown in Fig. \ref{fig:td}, the time-dependent safe set shrinks as we move backwards in time.  
A state at stage $t$ is feasible and is in the time-dependent risk-constrained safe set if a control value can be found that satisfies two constraints: $(i)$ $\cvar_{\alpha}[\dist(x_{t+1}, A) \mid \mathcal{F}_{t-1}]$ must be no greater than $\delta$ and $(ii)$ the future state must fall within the safe set in the next stage. 
The second constraint is strict and must hold for even the largest possible demand $w_{t}$. 
The state at stage $t$ must be large enough that even if we encounter the maximum demand $\max \{w^{(i)} \} > u_{\max} := 32$, which leads to $x_{t+1} = x_t + u_{\max} - \max \{w^{(i)} \} < x_t$, the future state remains within the safe set.  This leads to the minimum feasible state at stage $t$ being larger than the minimum feasible state at stage $t+1$.  By a similar logic, 
 the upper bound of $RS_{\alpha,\delta, t}(A)$ is no smaller at time $t$ than $t+1$ since $\min \{w^{(i)} \}>0$. In this case, however, the CVaR constraint $(i)$ is tighter than the constraint $(ii)$--the state at stage $t$ must be small enough to guarantee the probability of exceeding $A_{\max} := 100$ is small. We observe that the largest state for which the CVaR constraint is satisfied is constant across all of the stages. For this reason,  the upper bound of $RS_{\alpha,\delta, t}(A)$ is also constant across all stages in this example.

\begin{figure}[tb]
\begin{center}
\includegraphics[width=2.3in]{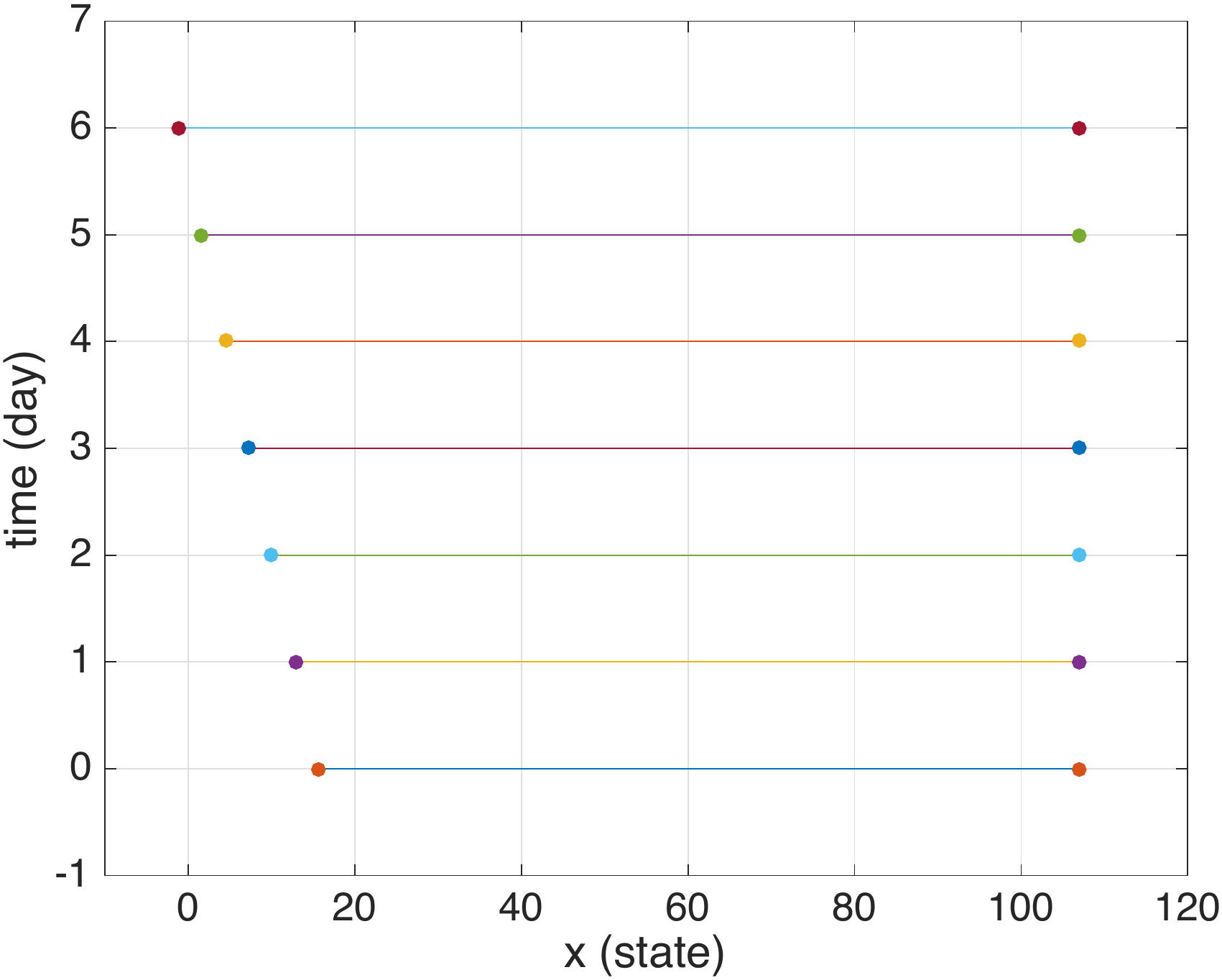}
\caption{Time-dependent risk-constraint safe set $RS_{\alpha, \delta, t}$ for $t= 0, \cdots, 6$.
}  
\label{fig:td}    
\end{center}           
\end{figure}

%
%
%
%
%
%
%
%
%
%

\section{Conclusions and Future Work}

A risk-based approach has been proposed for safety-aware optimal control of stochastic systems. 
We developed a computationally tractable dynamic programming solution, which provides a risk-constrained optimal controller and safe set.
 The latter can be used for verifying the safety of stochastic systems in a risk-constrained manner while enjoying useful connections with probabilistic safe sets.  
 We also identified the tradeoff between the risk tolerance and mean performance of our controller through a numerical example.
This approach can be extended in several interesting ways. To alleviate the dimensionality issue in our dynamic program- ming solution, it is worth exploring an occupation measure- based method. Furthermore, a distributionally robust control tool can immunize the proposed controller from potential errors in the probability distribution of disturbances.


\bibliographystyle{iEEEtran}

\bibliography{reference}

\begin{thebibliography}{10}
\providecommand{\url}[1]{#1}
\csname url@rmstyle\endcsname
\providecommand{\newblock}{\relax}
\providecommand{\bibinfo}[2]{#2}
\providecommand\BIBentrySTDinterwordspacing{\spaceskip=0pt\relax}
\providecommand\BIBentryALTinterwordstretchfactor{4}
\providecommand\BIBentryALTinterwordspacing{\spaceskip=\fontdimen2\font plus
\BIBentryALTinterwordstretchfactor\fontdimen3\font minus
  \fontdimen4\font\relax}
\providecommand\BIBforeignlanguage[2]{{%
\expandafter\ifx\csname l@#1\endcsname\relax
\typeout{** WARNING: IEEEtran.bst: No hyphenation pattern has been}%
\typeout{** loaded for the language `#1'. Using the pattern for}%
\typeout{** the default language instead.}%
\else
\language=\csname l@#1\endcsname
\fi
#2}}

\bibitem{Hu2005}
J.~Hu, M.~Prandini, and S.~Sastry, ``Aircraft conflict prediction in the
  presence of a spatially correlated wind field,'' \emph{IEEE Transactions on
  Intelligent Transportation Systems}, vol.~6, no.~3, pp. 326--340, 2005.

\bibitem{Mitchell2005}
I.~M. Mitchell and J.~A. Templeton, ``A toolbox of hamilton- jacobi solvers for
  analysis of nondeterministic continuous and hybrid systems,'' in
  \emph{International Workshop on Hybrid Systems: Computation and
  Control}.\hskip 1em plus 0.5em minus 0.4em\relax Springer, 2005.

\bibitem{Prajna2007}
S.~Prajna, A.~Jadbabaie, and G.~J. Pappas, ``A framework for worst-case and
  stochastic safety verification using barrier certificates,'' \emph{IEEE
  Transactions on Automatic Control}, vol.~52, no.~8, pp. 1415--1429, 2007.

\bibitem{Abate2008}
A.~Abate, M.~Prandini, J.~Lygeros, and S.~Sastry, ``Probabilistic reachability
  and safety for controlled discrete time stochastic hybrid systems,''
  \emph{Automatica}, vol.~44, no.~11, pp. 2724--2734, 2008.

\bibitem{Kariotoglou2017}
N.~Kariotoglou, M.~Kamgarpour, T.~H. Summers, and J.~Lygeros, ``The linear
  programming approach to reach-avoid problems for {Markov} decision
  processes,'' \emph{Journal of Artificial Intelligence Research}, accepted.

\bibitem{Yang2017aut}
I.~Yang, ``A dynamic game approach to distributionally robust safety
  specifications for stochastic systems,'' \emph{Automatica}, accepted.

\bibitem{Lesser2018}
K.~Lesser and A.~Abate, ``Multiobjective optimal control with safety as a
  priority,'' \emph{IEEE Transactions on Control Systems Technology}, accepted.

\bibitem{Svorenova2013}
M.~Svore\v{n}ov\'{a}, I.~\v{C}ern\'{a}, and C.~Belta, ``Optimal control of
  {MDPs} with temporal logic constraints,'' in \emph{Proceedings of the 52nd
  IEEE Conference on Decision and Control}, 2013.

\bibitem{Rockafellar2002a}
R.~T. Rockafellar and S.~Uryasev, ``Conditional value-at-risk for general loss
  distribution,'' \emph{Journal of Banking \& Finance}, vol.~26, pp.
  1443--1471, 2002.

\bibitem{Artzner1999}
P.~Artzner, F.~Delbaen, J.-M. Eber, and D.~Heath, ``Coherent measures of
  risk,'' \emph{Mathematical Finance}, vol.~9, no.~3, pp. 203--228, 1999.

\bibitem{Rockafellar2000}
R.~T. Rockafellar and S.~Uryasev, ``Optimization of conditional
  value-at-risk,'' \emph{Journal of Risk}, vol.~2, pp. 21--42, 2000.

\bibitem{Ruszczynski2006a}
A.~Ruszczy\'{n}ski and A.~Shapiro, ``Conditional risk mappings,''
  \emph{Mathematics of Operations Research}, vol.~31, no.~3, pp. 544--561,
  2006.

\bibitem{Ruszczynski2010}
A.~Ruszczy\'{n}ski, ``Risk-averse dynamic programming for markov decision
  processes,'' \emph{Mathematical Programming}, vol. 125, pp. 235--261, 2010.

\bibitem{Artzner2007}
P.~Artzner, F.~Delbaen, J.-M. Eber, D.~Heath, and H.~Ku, ``Coherent multiperiod
  risk adjusted values and {B}ellman's principle,'' \emph{Annals of Operations
  Research}, vol. 152, pp. 5--22, 2007.

\bibitem{Bauerle2011}
N.~B\"{a}uerle and J.~Ott, ``Markov decision processes with
  average-value-at-risk criteria,'' \emph{Mathematical Methods of Operations
  Research}, vol.~74, pp. 361--379, 2011.

\bibitem{Borkar2014}
V.~Borkar and R.~Jain, ``Risk-constrained {Markov} decision processes,''
  \emph{IEEE Transactions on Automatic Control}, vol.~59, no.~9, pp.
  2574--2579, 2014.

\bibitem{Chow2015}
Y.~Chow, A.~Tamar, S.~Mannor, and M.~Pavone, ``Risk-sensitive and robust
  decision-making: a {CVaR} optimization approach,'' in \emph{NIPS}, 2015.

\bibitem{Haskell2015}
W.~B. Haskell and R.~Jain, ``A convex analytic approach to risk-aware {M}arkov
  decision processes,'' \emph{SIAM Journal on Control and Optimization},
  vol.~53, no.~3, pp. 1569--1598, 2015.

\bibitem{Pflug2016}
G.~C. Pflug and A.~Pichler, ``Time-inconsistent multistage stochastic programs:
  {M}artingale bounds,'' \emph{European Journal of Operational Research}, vol.
  249, no.~1, pp. 155--163, 2016.

\bibitem{Miller2017}
C.~W. Miller and I.~Yang, ``Optimal control of conditional value-at-risk in
  continuous time,'' \emph{SIAM Journal on Control and Optimization}, vol.~55,
  no.~2, pp. 856--884, 2017.

\bibitem{Chow2013}
Y.~L. Chow and M.~Pavone, ``Stochastic optimal control with dynamic,
  time-consistent risk constraints,'' in \emph{Proceedings of 2013 American
  Control Conference}, 2013.

\bibitem{Cavus2014}
O.~\c{C}avu\c{s} and A.~Ruszczy\'{n}ski, ``Risk-averse control of undiscounted
  transient {M}arkov models,'' \emph{SIAM Journal on Control and Optimization},
  vol.~52, no.~6, pp. 3935--3966, 2014.

\bibitem{Krokhmal2002}
P.~Krokhmal, J.~Palmquist, and S.~Uryasev, ``Portfolio optimization with
  conditional value-at-risk objective and constraints,'' \emph{Journal of
  Risk}, vol.~4, pp. 43--68, 2002.

\bibitem{Robinson1996}
S.~M. Robinson, ``Analysis of sample-path optimization,'' \emph{Mathematics of
  Operations Research}, vol.~21, no.~3, pp. 513--528, 1996.

\bibitem{Summers2010}
S.~Summers and J.~Lygeros, ``Verification of discrete time stochastic hybrid
  systems: A stochastic reach-avoid decision problem,'' \emph{Automatica},
  vol.~46, no.~12, pp. 1951--1961, 2010.

\bibitem{Ding2013}
J.~Ding, M.~Kamgarpour, S.~Summers, A.~Abate, J.~Lygeros, and C.~Tomlin, ``A
  stochastic games framework for verification and control of discrete time
  stochastic hybrid systems,'' \emph{Automatica}, vol.~49, no.~9, pp.
  2665--2674, 2013.

\bibitem{Lesser2017}
K.~Lesser and M.~Oishi, ``Approximate safety verification and control of
  partially observable stochastic hybrid systems,'' \emph{IEEE Transactions on
  Automatic Control}, vol.~62, no.~1, pp. 81--96, 2017.

\bibitem{Hernandez2012}
O.~{Hern\'{a}ndez-Lerma} and J.~B. Lasserre, \emph{Discrete-Time Markov Control
  Processes: Basic Optimality Criteria}.\hskip 1em plus 0.5em minus 0.4em\relax
  Springer, 2012.

\end{thebibliography}
\vfill\eject

\end{document}